\documentclass[12pt,reqno]{amsart}

\usepackage{amssymb}
\usepackage[all,cmtip]{xy}

\usepackage{xcolor}

\usepackage{hyperref}

\newtheorem{theorem}{Theorem}[section]
\newtheorem{proposition}[theorem]{Proposition}
\newtheorem{lemma}[theorem]{Lemma}

\theoremstyle{remark}
\newtheorem{remark}[theorem]{Remark}

\numberwithin{equation}{section}

\newcommand{\TX}{\mathrm{T}X}
\newcommand{\Hun}{\mathrm{H}^1\hskip-2pt }
\newcommand{\Hze}{\mathrm{H}^0\hskip-2pt }
\newcommand{\KX}{\mathrm{K}_X}
\newcommand{\OXmD}{{\mathcal O}_X (-D)}
\newcommand{\OXD}{{\mathcal O}_X (D)}

\newcommand{\TXmD}{\mathrm{T}{X}(-D)}

\newcommand{\TXD}{\mathrm{T}X (-D)}
\newcommand{\Atlog}{\mathrm{At}_{D}(E)}

\newcommand{\Epar}{E_\star}

\newcommand{\cEpar}{{\mathcal E}_\star}
\newcommand{\EndpEpar}{\mathrm{End}_p\hskip-2pt \left(\Epar\right)}
\newcommand{\EndpFEpar}{\mathrm{End}_p^F\hskip-2pt \left(\Epar\right)}
\newcommand{\EndpLEpar}{\mathrm{End}_p^L\hskip-2pt \left(\Epar\right)}
\newcommand{\EndnLEpar}{\mathrm{End}_n^L\hskip-2pt \left(\Epar\right)}
\newcommand{\Lpar}{L_\star}
\newcommand{\Fpar}{F_\star}
\newcommand{\AtthpE}{{\rm At}^L_p (E)}
\newcommand{\AtpE}{{\rm At}_p (E)}
\newcommand{\EndE}{\mathrm{End}(E)}
\newcommand{\EndpE}{\mathrm{End}_p\hskip-1pt \left(\Epar\right) }
\newcommand{\EndzepE}{\mathrm{End}^0_p \hskip-1pt \left(\Epar\right) }
\begin{document}

\baselineskip=15pt

\title[Isomonodromic deformations and stable parabolic bundles]{Isomonodromic deformations 
of logarithmic connections and stable parabolic vector bundles}

\author[I. Biswas]{Indranil Biswas}

\address{School of Mathematics, Tata Institute of Fundamental
Research, Homi Bhabha Road, Mumbai 400005, India}

\email{indranil@math.tifr.res.in}

\author[V. Heu]{Viktoria Heu}

\address{Institut de Recherche Math\'ematique Avanc\'ee, Univesit\'e de Strasbourg, 7 rue 
Ren\'e-Descartes, 67084 Strasbourg Cedex, France}

\email{heu@math.unistra.fr}

\author[J. Hurtubise]{Jacques Hurtubise}

\address{Department of Mathematics, McGill University, Burnside
Hall, 805 Sherbrooke St. W., Montreal, Que. H3A 0B9, Canada}

\email{jacques.hurtubise@mcgill.ca}



\date{}

\begin{abstract}
We consider irreducible logarithmic connections $(E,\,\delta)$ over compact Riemann surfaces $X$ 
of genus at least two. The underlying vector bundle $E$ inherits a natural parabolic structure 
over the singular locus of the connection $\delta$; the parabolic structure is
given by the residues of $\delta$. We prove that for the universal isomonodromic 
deformation of the triple $(X,\,E,\,\delta)$, the parabolic vector bundle corresponding to a 
generic parameter in the Teichm\"uller space is parabolically stable. In the case of parabolic vector 
bundles of rank two, the general parabolic vector bundle is even parabolically very stable. 
\end{abstract}

\maketitle
 
\tableofcontents 

\section{Introduction}

Let $(X,\,D)$ be a compact Riemann surface of genus $g$ with $n$ (ordered) marked points 
$D\,=\,(x_1,\,\cdots,\, x_n)$. The monodromy functor produces an equivalence between the category of 
holomorphic connections $(E_0\, ,\delta_0)$ on $X\setminus D$ and the category of equivalence 
classes of linear representations of $\pi_1(X\setminus D,\, x_0)$.
Here the morphisms are 
isomorphisms of vector bundles with connections on one side and conjugation of representations on the 
other side; this is an
example of Riemann--Hilbert correspondence. Moreover, given $(E_0,\, \delta_0)$ , there exists a logarithmic connection 
$(E,\,\delta)$ on $X$, singular over $D$, which extends $(E_0,\, \delta_0)$. Indeed, one can 
choose for example a Deligne extension \cite{De}.

The classical Riemann-Hilbert problem takes $X$ to be the projective line ${\mathbb C}{\mathbb P}^1$ and asks whether 
it is possible to choose $(E,\, \delta)$ extending $(E_0,\, \delta_0)$ such that $E$ is the 
trivial holomorphic vector bundle over $X\,=\, {\mathbb C}{\mathbb P}^1$. The answer to it is no in general; the first 
counterexample was constructed by Bolibruch in \cite{Bolibruch1}. However, the Riemann-Hilbert problem
is known to have a positive answer when $\text{rank}(E_0)\,=\,2$, or when the connection $\delta_0$ is irreducible
\cite{Plemelj}, \cite{Dekkers},\cite{Bolibruch2}, \cite{Kostov}.

An appropriate formulation for the classical Riemann-Hilbert problem in higher genus is to ask whether $(E\, ,\delta)$ can be chosen such that $E$ is semistable of degree $0$. Indeed, with that formulation, the general negative answer as well as the sufficient conditions for positive answers remain valid, as proven in \cite{Helene1} and \cite{Helene2}.

On the other hand, the fundamental group $\pi_1(X\setminus D,\, x_0)$ does not depend on the 
complex structure of $X$. Let us consider $(X,\,D)$ as a fiber of 
the universal family of curves over the Teichm\"uller space $\mathcal{T}_{g,n}$ of
genus $g$ surfaces with $n$ marked points: 
$$\renewcommand{\labelstyle}{\textstyle} \xymatrix{ 
(X,D)\ar[r]\ar[d]&(\mathcal{X},\mathcal{D})\ar[d]_p\\\{t_0\}\ar[r]&\mathcal{T}_{g,n}}$$ The 
fundamental group of each punctured fiber can be identified with $\pi_1(X\setminus D,\, 
x_0)$, because $\mathcal{T}_{g,n}$ is contractible. Given any $(E,\, \delta)$ on $(X,\,D)$, it
extends to a flat logarithmic connection 
$(\mathcal{E},\,\delta')$ over $\mathcal{X}$, singular over $\mathcal{D}$; this
flat logarithmic connection $(\mathcal{E},\,\delta')$ is called the 
\emph{(universal) isomonodromic deformation } of $(E,\, \delta)$ (see Section 
\ref{sec:univdef}). It is called isomonodromic because with respect to a convenient 
identification of the fundamental group of the fibers, the corresponding family of monodromy 
representations is constant.

We are led to another Riemann--Hilbert type problem: \emph{Given any $(E,\, \delta)$, is 
there a parameter $t \,\in\, \mathcal{T}_{g,n}$ such that for the logarithmic connection 
$(\mathcal{E}^t,\delta^t)$ on $p^{-1}(t)$ induced by the isomonodromic deformation, the 
vector bundle $\mathcal{E}^t$ is semistable?} The partial answers in \cite{Bolibruch3} and 
\cite{Heu2} to this question were generalized in \cite{BHH1} to the following. If the genus 
$g$ of $X$ is at least $2$ and $\delta$ is irreducible, then for generic parameters $t\,\in\, 
\mathcal{T}_{g,n}$, the vector bundle $\mathcal{E}^t$ is not only semistable but stable.
In case rank two, the general vector bundle is even very stable \cite{BHH2}. This 
remains valid, for an appropriate generalization of the universal isomonodromic deformation 
in case $\delta$ has irregular singularities \cite{Heu2}, \cite{BHH3}.

\begin{remark}
Note that the degree of the vector bundle is a topological invariant and thus remains
constant along the deformation. If one wishes to investigate the above question in the case 
$(E,\delta)$ is \emph{reducible, i.e.,} there is a subbundle $0\,\subsetneq \,F\,\subsetneq\, E$ 
preserved by $\delta$, then one has to impose that $F$ is not a destabilizing bundle. Under 
this additional assumption, the proof in \cite{BHH1} still applies.
\end{remark}

On the other hand, given a logarithmic connection $(E,\,\delta)$ on a curve, there is a 
natural parabolic structure on $E$ supported by the singularities of the connection such that 
the parabolic structure at a singular point of the connection is given by the residue of the 
logarithmic connection at that point (see Section \ref{se2.3}). Therefore, underlying the 
universal isomonodromic deformation is also a family of parabolic vector bundles parametrized 
by $\mathcal{T}_{g,n}$. Our aim here is to investigate the above questions on stability 
and very stability of the general underlying bundle in this context of parabolic vector bundles (see 
Sections \ref{se2.2} and \ref{sec:verystable}).

We prove the following result in two steps 
(see Theorem \ref{thm1} and Theorem \ref{thm2}).

\begin{theorem} 
Let $X$ be a compact Riemann surface of genus $g\geq 2$, and let $D$ be a divisor on $X$. 
Let $\delta$ be a logarithmic connection, singular over $D$, on a holomorphic vector
bundle $E\,\longrightarrow\, X$.
Let $(\mathcal{E},\delta')$ be its universal isomonodromic deformation, with
$$
\mathcal{E}\, \longrightarrow\, {\mathcal X}
\, \stackrel{p}{\longrightarrow}\,{\mathcal T}_{g,n}\, .
$$ 
Denote $\mathcal{E}^t\,:=\,\mathcal{E}|_{\mathcal{X}_t}$, where
$\mathcal{X}_t\,:= \,p^{-1}\left(t\right)$. Denote by $\cEpar^t$ the corresponding parabolic
vector bundle over $\mathcal{X}_t$ with parabolic structure induced by $\delta'|_{\mathcal{E}^t}$.
Then there are closed analytic subsets $${\mathcal Y}\, \subset \, {\mathcal Y}'\, \subset\,
{\mathcal Y}''\, \subset\, \mathcal{T}_{g,n}$$ such that the following statements hold:
\begin{itemize}
\item[$\bullet$] for every $t\, \in\, {\mathcal T}_{g,n}\setminus {\mathcal Y}$, the parabolic
vector bundle $\cEpar^t$
is parabolically semistable;
\item[$\bullet$] for every $t\, \in\, {\mathcal T}_{g,n}\setminus {\mathcal Y}'$, the parabolic
vector bundle $\cEpar^t$
is parabolically stable;
\item[$\bullet$] for every $t\, \in\, {\mathcal T}_{g,n}\setminus {\mathcal Y}''$, the parabolic
vector bundle $\cEpar^t$
is parabolically very stable.
\end{itemize}
If $\delta$ is irreducible, then the analytic subsets ${\mathcal Y}$ and ${\mathcal Y'}$ of ${\mathcal T}_{g,n}$ are
proper, and their codimensions are bounded as follows
$$\mathrm{codim}\left({\mathcal Y}\right)\geq g\, ; \quad \mathrm{codim}\left({\mathcal Y}'\right)\geq g-1\, .$$
If $\delta$ is irreducible and $E$ is of rank $2$, then the analytic subset ${\mathcal Y}''$ is also proper. 
\end{theorem}

The proof is similar to the non-parabolic case treated in \cite{BHH1} and \cite{BHH2}: the fact that the sets $ 
{\mathcal Y}, {\mathcal Y}', {\mathcal Y}''\subset {\mathcal T}_{g,n}$ are analytically closed is known from 
\cite{Nitsure}. The main issue is proving that these are proper subsets. We proceed with a 
deformation-theoretic approach.
 
This paper is the final one in a series examining the behaviour of ``generic properties'' 
such as stability under isomonodromic deformation; the general gist is that isomonodromic 
deformation is in some sense transversal to the unstable locus. In previous papers, the 
connection was also allowed to have singularities, but these were basically independent of 
the structure examined. in the set-up considered here, the parabolic structure and the 
singularities of the connection are intertwined; the genericity result still holds, however.
 
\section{Logarithmic connections and parabolic bundles}

In this section, we recall the definition of the Atiyah bundle for a vector bundle over a 
pointed curve, and how the Atiyah exact sequence can be used to define logarithmic connections 
on the vector bundle on the one hand, and infinitesimal deformations of the vector bundle on 
the pointed curve on the other hand. We further recall that if a vector bundle is endowed with 
a logarithmic connection, then it has a natural parabolic structure defined by the residues of 
the connection.

\subsection{Logarithmic connections and the Atiyah bundle}\label{Sec:Atiyah}
 
Let $X$ be a compact connected Riemann surface of genus $g$, with $g\, \geq\, 2$.
Fix a finite nonempty subset $$D\,=\,\{x_1,\, \cdots,\, x_n\}\, \subset\, X$$
of distinct ordered points of cardinality $n\,\geq\, 1$. We will employ the
convention of denoting by $\mathrm{T}Z$ the holomorphic tangent
bundle of a complex manifold $Z$. Let
$$
\TX\left(-\log D\right)\,=\, \TXD\,:=\,\TX\otimes_{{\mathcal O}_X} \OXmD
$$
be the logarithmic tangent bundle of $X$.

Take a holomorphic vector bundle $E$ over $X$ of rank $r$. For any $i\, \geq\, 0$, let
$\text{Diff}^i\left(E,\,E\right)$ be the holomorphic vector bundle on $X$ defined by the sheaf of
holomorphic differential operators, of order at most $i$, from the sheaf of holomorphic
sections of $E$ to itself. In other words,
$$
\text{Diff}^i\left(E,\,E\right)\,=\, \text{Hom}(J^i(E),\, E)\,=\, E\otimes J^i(E)^\vee\, ,
$$
where $J^i(E)$ it the $i$-th jet bundle for $E$. Consider the symbol homomorphism
\begin{equation}\label{e1}
\sigma_1\, :\, \text{Diff}^1\left(E,\,E\right)\, \longrightarrow\, \TX\otimes \EndE\, .
\end{equation}
We recall the construction of $\sigma_1$. Take any $x\, \in\, X$ and any $w\, \in\, \mathrm{T}^\vee_xX$. Let
$f_w$ be a holomorphic function defined around $x$ such that $f_w\left(x\right)\,=\, 0$ and $df_w\left(x\right)\,=\, w$.
Let ${\mathcal D}_x$ be a holomorphic section of $\text{Diff}^1\left(E,\,E\right)$ defined around $x$. Then
for any $v\, \in\, E_x$, we have
\begin{equation}\label{f1}
w\left(\sigma_1\left({\mathcal D}_x\left(x\right)\right)\left(v\right)\right)\,=\,
{\mathcal D}_x \left(f_w\cdot v'\right)(x)\, ,
\end{equation}
where $v'$ is a holomorphic section of
$E$ defined around $x$ such that $v'\left(x\right)\,=\, v$; note that both sides of \eqref{f1}
are elements of $E_x$. The homomorphism $\sigma_1$ is evidently surjective.
The logarithmic Atiyah bundle is defined as
$$
\Atlog\, :=\, \sigma^{-1}_1\left(\TXD \otimes
\text{Id}_E\right)\, \subset\, \text{Diff}^1\left(E,\,E\right)\, .
$$
It fits in the logarithmic Atiyah exact sequence
\begin{equation}\label{e2}
0\, \longrightarrow\, \EndE\, \longrightarrow\,
\Atlog\, \stackrel{\sigma}{\longrightarrow}\, \TXD \, \longrightarrow\, 0\, ,
\end{equation}
where $\sigma$ is the restriction of the symbol homomorphism $\sigma_1$ in \eqref{e1}. Therefore,
a holomorphic section of $\Atlog$ over an open subset $U\, \subset\, X$ is a
holomorphic differential operator
\begin{equation}\label{oc}
D_U\, :\, E\vert_U \, \longrightarrow\, (E\otimes K_X\otimes {\mathcal O}_U(D))\vert_U\, ,
\end{equation}
where $K_X\, =\, (\TX)^*$ is the holomorphic cotangent bundle of $X$,
satisfying the following Leibniz condition:
$$
D_U (f\cdot s)\,=\, f\cdot D_U(s) + s\otimes df
$$
for every holomorphic
function $f_U$ on $U$ and every holomorphic section $s$ of $E$ over $U$.

We recall that a logarithmic connection on $E$ singular over $D$ is a holomorphic splitting
of the exact sequence in \eqref{e2}, meaning a holomorphic homomorphism
$$
\delta\, :\, \TXD \, \longrightarrow\, \Atlog
$$
such that $\sigma\circ \delta\,=\, \text{Id}_{\TXD}$, where $\sigma$
is the homomorphism in \eqref{e2} \cite{De} (see also \cite{At}).

So a logarithmic connection $\delta$ on $E$ singular over $D$ corresponds to a
holomorphic differential operator over $X$
$$D_X\, :\, E\, \longrightarrow\, E\otimes K_X\otimes {\mathcal O}_U(D)$$
as in \eqref{oc} satisfying the Leibniz condition.

We have the following:
 
\begin{enumerate}
\item The infinitesimal deformations of the $n$-pointed compact Riemann surface
$\left(X,\, D\right)$ are parametrized by $\Hun\left(X,\, \TXD\right)$.

\item The infinitesimal deformations of the above triple $\left(X,\, D,\, E\right)$
are parametrized by $\Hun\left(X,\, \Atlog\right)$.

\item The map $\Hun\left(X,\, \TXD\right) \,\longrightarrow\, \Hun\left(X,\, \Atlog\right)$
corresponding to isomonodromic deformation is the one induced by the connection
$\delta\, :\, \TXD \, \longrightarrow\, \Atlog$.
\end{enumerate}
Here $(1)$ is standard, $(2)$ is a consequence of the results in \cite{Huang} and $(3)$ is explained in \cite{BHH1}.

\subsection{Residue of a logarithmic connection}\label{sec:Res}

Take any $x_j\,\in\, D$. There is a canonical homomorphism
\begin{equation}\label{e3}
\phi_j\, :\, {\Atlog}_{x_j}\,\longrightarrow\, \EndE_{x_j}\,=\, \text{End}\left(E_{x_j}\right)
\end{equation}
which we will now describe. Consider the commutative diagram of homomorphisms of vector spaces
\begin{equation}\label{cd0}\renewcommand{\labelstyle}{\textstyle}
\xymatrix{
0 \ar[r]& \text{End}\left(E_{x_j}\right)\ar[r]^{\alpha_j} \ar@{=}[d] &
{\Atlog}_{x_j} \ar[rr]^{\sigma\left(x_j\right)} \ar[d]^a && {\TXD}_{x_j}\ar[r] \ar[d]^b& 0\\
0 \ar[r] & \text{End}\left(E_{x_j}\right) \ar[r]^{c_j\ \ } & \text{Diff}^1\left(E,\,E\right)_{x_j} \ar[rr]^{\sigma_1\left(x_j\right)\ \ \ } && \left(\TX\otimes \EndE\right)_{x_j} \ar[r] & 0\, , 
 }
\end{equation}
where $\sigma$ and $\sigma_1$ are the homomorphisms in \eqref{e2} and \eqref{e1}
respectively, and the top exact row is the restriction of the exact sequence in \eqref{e2}
to the point $x_j$ while the bottom exact row is the restriction of the Atiyah exact sequence
to the point $x_j$; both the rows in \eqref{cd0} are exact. The homomorphism $a$ in \eqref{cd0} is given
by the natural inclusion of the coherent sheaf $\Atlog$ in $\text{Diff}^1\left(E,\,E\right)$,
while $b$ is induced by $a$.
Note that $b\,=\,0$, as $x_j$ is a point of $D$. This implies that $\sigma_1\left(x_j\right)\circ a\,=\, b\circ \sigma\left(x_j\right)\,=\,
0$. Now from the exactness of the bottom row in \eqref{cd0}
it follows that $\text{image}\left(a\right)\, \subset\,
\text{image}\left(c_j\right)$, and hence there is a unique homomorphism
$$\phi_j\, :\, {\Atlog}_{x_j}\,\longrightarrow\, \text{End}\left(E_{x_j}\right)$$
such that $a\,=\, c_j\circ\phi_j$. This produces the homomorphism in \eqref{e3}.

{}From the commutativity of the diagram in \eqref{cd0}
we conclude that $\phi_j\circ \alpha_j$ coincides with the
identity map of $\text{End}\left(E_{x_j}\right)$. From this it follows immediately
that the restriction of $\sigma\left(x_j\right)$ to $$\text{kernel}\left(\phi_j\right)\, \subset\,
{\Atlog}_{x_j}$$ is an isomorphism with ${\TXD}_{x_j}$. Using this
isomorphism of $\text{kernel}\left(\phi_j\right)$ with ${\TXD}_{x_j}$
we have a decomposition
\begin{equation}\label{e4}
{\Atlog}_{x_j}\,=\, \text{End}\left(E_{x_j}\right)\oplus \text{kernel}\left(\phi_j\right)\,=\,
\text{End}\left(E_{x_j}\right)\oplus {\TXD}_{x_j}\, .
\end{equation}
The fiber ${\TXD}_{x_j}$ is identified
with $\mathbb C$ using the Poincar\'e adjunction formula \cite[p.~146]{GH}. More
explicitly, for any holomorphic coordinate $z$ around $x_j$ with $z\left(x_j\right)\,=\, 0$, the evaluation
of the section $z\frac{\partial}{\partial z}$ of $\TXD$ at the point $x_j$
is independent of the choice of the holomorphic 
coordinate function $z$; the above identification between ${\TXD}_{x_j}$
and $\mathbb C$ sends this independent element of ${\TXD}_{x_j}$ to $1\,\in\, \mathbb C$.

Let $\delta\, :\, \TXD \, \longrightarrow\, \Atlog$ be a logarithmic connection on $E$
singular over $D$. For any $x_j\, \in\, D$, consider
\begin{equation}\label{d0}
\delta\left(x_j\right)\left(1\right)\, \in\, {\Atlog}_{x_j}\,=\, \text{End}\left(E_{x_j}\right)\oplus\mathbb C\, ;
\end{equation}
here the above identification ${\TXD}_{x_j}\,=\, \mathbb C$ is being used. Let
\begin{equation}\label{e5}
\text{Res}\left(\delta\right)\left(x_j\right)\, \in \, \text{End}\left(E_{x_j}\right)
\end{equation}
be the component of $\delta\left(x_j\right)\left(1\right)$ in the direct
summand $\text{End}\left(E_{x_j}\right)$ in \eqref{d0}. This endomorphism
$\text{Res}\left(\delta\right)\left(x_j\right)$ is called the {\it residue} of $\delta$ at the point $x_j$.

The residue
is called \emph{resonant} if it admits two eigenvalues whose difference is a non-zero integer.
The connection $\delta$ is said to be resonant if it possesses a resonant residue.

Let $D_X\, :\, E\, \longrightarrow\, E\otimes K_X\otimes {\mathcal O}_X(D)$ be a
holomorphic differential operator over $X$ as in \eqref{oc} associated to a logarithmic
connection $\delta$ on $E$. For any point $x_j\, \in\, D$, consider the composition
$$
E\, \stackrel{D_X}{\longrightarrow}\, E\otimes K_X\otimes {\mathcal O}_X(D)\, \longrightarrow\,
(E\otimes K_X\otimes {\mathcal O}_X(D))_{x_j}\, =\, E_{x_j}\, ;
$$
the fiber $(K_X\otimes {\mathcal O}_X(D))_{x_j}$ is identified with $\mathbb C$ using the
Poincar\'e adjunction formula. This composition is ${\mathcal O}_X$--linear, and hence it
produces an endomorphism $R_j\, \in\, \text{End}(E_{x_j})$. This endomorphism $R_j$ coincides with
the residue $\text{Res}\left(\delta\right)\left(x_j\right)$ in \eqref{e5}.

\subsection{Parabolic bundles and the notion of stability}\label{se2.2}

Let $E$ be a holomorphic vector bundle over $X$ of positive rank.
A \textit{quasiparabolic structure} on $E$ over the divisor $D$ is a
strictly decreasing filtration of subspaces
\begin{equation}\label{qf}
E_{x_j}\,=\, E^1_j\, \supsetneq\, E^2_j \,\supsetneq\, \cdots\, \supsetneq\,
E^{n_j}_j \, \supsetneq\, E^{n_j+1}_j \,=\, 0
\end{equation}
for every $1\, \leq\, j\, \leq\, n$. A \textit{parabolic structure} on $E$ over $D$ is a
quasiparabolic structure as above together with $n$ decreasing sequences of real numbers
$$
0\, \leq\, \alpha^1_j\, <\, \alpha^2_j\, <\,
\cdots\, < \, \alpha^{n_j}_j \, < 1\, , \ \ 1\, \leq\, j\, \leq\, n \, ;
$$
the real number $\alpha^i_j$ is called the parabolic weight of the subspace $E^i_j$ in
the quasiparabolic filtration. The multiplicity of a parabolic weight
$\alpha^i_j$ at $x_j$ is defined to be the dimension of the complex
vector space $E^i_j/E^{i+1}_j$.
A parabolic vector bundle is a vector bundle with a
parabolic structure. We shall refer to the collection of weights and respective multiplicities at each puncture as the \emph{parabolic data} of a parabolic vector bundle. More details on parabolic bundles can be found in \cite{MS}, \cite{MY}.

Let $\Epar\,=\, \left(E,\, \{E^i_j\}, \,\{\alpha^i_j\}\right)$ be a parabolic bundle as above.
The \textit{parabolic degree} of $\Epar$ is defined to be
$$
\text{par-deg}\left(\Epar\right)\,=\, \text{degree}(E)+\sum_{j=1}^n \sum_{i=1}^{n_j}
\alpha^i_j \dim \left(E^i_j/E^{i+1}_j\right)
$$
\cite[p.~214, Definition~1.11]{MS}, \cite[p.~78]{MY}.

Take any holomorphic subbundle $F\, \subset\, E$. For each $x_j\, \in\, D$, the
fiber $F_{x_j}$ has a filtration obtained by intersecting the quasiparabolic filtration
of $E_{x_j}$ with the subspace $F_{x_j}$. The parabolic weight of a subspace $S\, \subset\,
F_{x_j}$ in this filtration is the maximum of the numbers
$$\{\alpha^i_j\, \mid\, S\, \subset \, E^i_j\cap F_{x_j}\}\, .$$
This way, the parabolic structure on $E$ produces a parabolic structure on the subbundle $F$.
The resulting parabolic bundle will be denoted by $\Fpar$.

A parabolic vector bundle $\Epar\,=\, \left(E,\, \{E^i_j\}, \,\{\alpha^i_j\}\right)$ is called
\textit{stable} (respectively, \textit{semistable}) if for all subbundles $F\, \subsetneq
\, E$ of positive rank the inequality
$$
\frac{\text{par-deg}\left(\Fpar\right)}{\text{rank}\left(\Fpar\right)}\, <\, 
\frac{\text{par-deg}\left(\Epar\right)}{\text{rank}\left(\Epar\right)}\ \
\text{\Big (respectively,} \ \frac{\text{par-deg}\left(\Fpar\right)}{\text{rank}\left(\Fpar\right)} \, \leq\,
\frac{\text{par-deg}\left(\Epar\right)}{\text{rank}\left(\Epar\right)}\text{\Big )}
$$
holds \cite{MS}.

\subsection{Parabolic structure from a logarithmic connection}\label{se2.3}

Let $$\delta\, :\, \TXD \, \longrightarrow\, \Atlog$$ be a logarithmic connection on $E$,
singular over $D$.
Using the residues of $\delta$ defined in \eqref{e5}, we will construct a parabolic structure on $E$.
To each eigenvalue $\boldsymbol{\lambda}$ of 
$\text{Res}\left(\delta\right)\left(x_j\right)$, we associate $$\lambda:=\{\Re \left(\boldsymbol{\lambda}\right)\}:=\Re \left(\boldsymbol{\lambda}\right) - \lfloor \Re \left(\boldsymbol{\lambda}\right) \rfloor \, \in [0,1[\, , $$
the fractional part of its real part. 
Let $x_j\, \in\, D$ and let 
$$
0\, \leq\, \lambda^1_j\, <\, \lambda^2_j\, <\,
\cdots\, < \, \lambda^{n_j}_j \, < 1
$$
be the fractional parts of the real parts of the eigenvalues of
$\text{Res}\left(\delta\right)\left(x_j\right)$. Let $F^i_j\, \subset\, E_{x_j}$
be the sum of the generalized eigenspaces corresponding to those eigenvalues
$\boldsymbol{\lambda}$ of $\text{Res}\left(\delta\right)\left(x_j\right)$ such that
$\{\Re \left(\boldsymbol{\lambda}\right)\}\,=\,\lambda^i_j$.
The parabolic weights of $E$ at $x_j$ are the eigenvalues
$\{\lambda^i_j\}_{i=1}^{n_j}$. The subspace of $E_{x_j}$ corresponding to
the parabolic weight $\lambda^i_j$ is $\bigoplus_{k\geq i} F^k_j$. 
Note that according to this definition, the parabolic structure at $x_i$ is determined
by the semisimple part $\text{Res}^{ss}\left(\delta\right)\left(x_j\right)$ (with respect to
the Jordan decomposition) of the residue at $x_i$. If $$\Re \left(\boldsymbol{\lambda}\right)
\,=\, \{\Re \left(\boldsymbol{\lambda}\right)\}\,\in\, [0,1[$$ for each eigenvalue
for each residue of $\delta$, then $\delta$ is the called the Deligne extension of the restriction of $\delta$ to $E|_{X\setminus D}$.
 
\begin{remark}We note that $\text{degree}(E)+\sum_{j=1}^n\mathrm{trace}\left(\mathrm{Res}\left(\delta\right)\left(x_j\right)\right)=0$ \cite[p.~16, Theorem~3]{Oh}. Therefore,
\begin{equation}\label{de}
\text{par-deg}\left(\Epar\right)\,:=\,
\text{degree}(E)+\sum_{j=1}^n\sum_{i=1}^{n_j} \lambda^i_j\, \in \, \mathbb{Z}\, , 
\end{equation}
 where $\Epar$ is the parabolic vector bundle constructed from
$\left(E,\, \delta\right)$.\end{remark}

\section{Infinitesimal deformations of parabolic bundles}

We shall now establish the space of infinitesimal deformations of parabolic bundles on pointed 
curves, where the base is allowed to vary. Moreover, we are going to take into account the 
information of a further subbundle, which shall later be used for testing of parabolic 
stability.

\subsection{Infinitesimal deformations with fixed base curve}\label{sec:fixed curve}

Fix a pair $\left(X,\, D\right)$ as before. Let $\Epar\,=\, \left(E,\, \{E^i_j\}, \,\{\alpha^i_j 
\}\right)$ be a parabolic vector bundle on $X$ with parabolic structure over the divisor $D$. Let
\begin{equation}\label{eq:defEndpE}
\EndpE\, \subset\, \EndE\,=\, E\otimes E^\vee
\end{equation}
denote the coherent subsheaf that preserves the quasiparabolic filtration over every point of $D$.
So, $\EndpE$ coincides with $\EndE$ over the complement $X\setminus D$.
For each point $x_j\, \in\, D$, the image
of $\text{End}_p(E)_{x_j}$ in $\text{End}\left(E_{x_j}\right)\,=\, \EndE_{x_j}$
consists of all endomorphisms that preserve the quasiparabolic filtration over $x_j$.
In other words, for a section $s$ of $\EndpE$, we have
$$
s\left(E^i_j\right)\, \subseteq\, E^{i}_j
$$
for all $x_j$ in the domain of definition of $s$ and all $1\,\leq\, i\,\leq\, n_j$ (as in \eqref{qf}). Let
\begin{equation}\label{ei}
\text{End}_{p,j}(E)\, \subset\, \text{End}\left(E_{x_j}\right)
\end{equation}
be the image of $\text{End}_p(E)_{x_j}$ in $\text{End}\left(E_{x_j}\right)$. We
have a short exact sequence of coherent sheaves on $X$
\begin{equation}\label{e6}
0\,\longrightarrow\, \EndpE\,\stackrel{\beta_0}{\longrightarrow}\,
\EndE\,\longrightarrow\,\bigoplus_{j=1}^n \EndE_{x_j}/\text{End}_{p,j}(E)
\,\longrightarrow\, 0\, .
\end{equation}
It is known that the infinitesimal deformations of $\Epar$ are parametrized by
$\Hun\left(X,\, \EndpE\right)$ \cite[Section~5]{Yo}. 

\subsection{Infinitesimal deformations with varying base curve}\label{sec:varbase}

Consider the homomorphism $\phi_j$ constructed in \eqref{e3}. The composition
$$
{\Atlog}_{x_j}\,\stackrel{\phi_j}{\longrightarrow}\, \EndE_{x_j}
\,\longrightarrow\, \EndE_{x_j}/\text{End}_{p,j}(E)
$$
will be denoted by $\widehat{\phi}_j$; the above map
$\EndE_{x_j}\,\longrightarrow\, \EndE_{x_j}/\text{End}_{p,j}(E)$ is the
quotient by the subspace in \eqref{ei}. Note that this composition homomorphism
is surjective. Let
$$
\AtpE\,\subset\, {\Atlog}
$$
be the coherent subsheaf that fits in the following short exact sequence:
\begin{equation}\label{e7}
0\,\longrightarrow\, \AtpE\,\longrightarrow\, {\Atlog}\,
\stackrel{\oplus_j \widehat{\phi}_j}{\longrightarrow}\,\bigoplus_{j=1}^n
\EndE_{x_j}/\text{End}_{p,j}(E)\,\longrightarrow\, 0\, .
\end{equation}
Therefore, using \eqref{e2} we have the following commutative diagram with exact rows and columns:\begin{equation}\label{e8}
\renewcommand{\labelstyle}{\textstyle}
\xymatrix{
& 0\ar[d] & 0\ar[d]\\
 0 \ar[r]& \EndpE\ar[d]^{\beta} \ar[r] & \EndE\ar[d]
\ar[r] & \bigoplus_{j=1}^n
\EndE_{x_j}/\text{End}_{p,j}(E) \ar@{=}[d] \ar[r]& 0\\
 0 \ar[r] & \AtpE \ar[r]\ar[d]^{\sigma'} & {\Atlog}\ar[d]^{\sigma} \ar[r] & \bigoplus_{j=1}^n
\EndE_{x_j}/\text{End}_{p,j}(E)\ar[r]& 0\\
& \TXD \ar@{=}[r] \ar[d] & \TXD\ar[d] \\
 & 0 & 0
 }
\end{equation}
 where $\sigma'$ is the restriction of $\sigma$ in \eqref{e2}. We note that a holomorphic
section of $\AtpE$ over an open subset $U\, \subset\, X$ is a
holomorphic differential operator of order one
$$
D_U\, :\, E\vert_U \, \longrightarrow\, E\vert_U
$$
satisfying the following conditions:
\begin{itemize}
\item the symbol of $D_U$ is a holomorphic section of $\TXD$ over $U$ (so
$D_U$ is a section of ${\Atlog}$ over $U$), and

\item for every holomorphic section $s$ of $E\vert_U$, and
every $x_j\, \in\, D\cap U$, if $s\left(x_j\right)\, \in\, E^i_j\, \subset\, E_{x_j}$, then
$D_U\left(s\right)\left(x_j\right)\, \in\, E^i_j$. Here we used the notation in \eqref{qf}.
\end{itemize}

\begin{lemma}\label{lem1}
The infinitesimal deformations of the triple $\left(X,\, D,\, \Epar\right)$, with parabolic
data of fixed type (fixed parabolic weights and their multiplicities), are parametrized by
$\Hun\left(X,\, \AtpE\right)$. The homomorphism
$$
\beta_*\, :\, \Hun\left(X,\, {\rm End}_p\left(\Epar\right)\right)\,\longrightarrow\,
\Hun\left(X,\, \AtpE\right)\, ,
$$
induced by $\beta$ in \eqref{e8}, corresponds to the map of infinitesimal deformations where
the pair $\left(X,\, D\right)$ is kept fixed. The homomorphism
$$
\sigma'_*\, :\, \Hun\left(X,\, \AtpE\right)\,\longrightarrow\,\Hun\left(X,\, \TXD\right)\, ,
$$
induced by $\sigma'$ in \eqref{e8}, is the forgetful map that
sends any infinitesimal deformation of $\left(X,\, D,\, \Epar\right)$ to the
infinitesimal deformation of $\left(X,\, D\right)$ obtained by simply forgetting $\Epar$.
\end{lemma}

\begin{proof}
This lemma is standard.
Consider the sheaf of groups on $X$ given by the local automorphisms of $E$ that preserve the parabolic structure (this
means that the quasiparabolic structure is preserved, because the parabolic weights do not move).
The corresponding sheaf of Lie algebras is ${\rm End}_p\left(\Epar\right)$. More generally consider the
sheaf of groups on $X$ given by the local automorphisms of the pair $(X,\, E)$ that preserve the parabolic structure.
The corresponding sheaf of Lie algebras is $\AtpE$. The lemma can be derived from these
observations.
\end{proof}

The homomorphism $$\Hun\left(X,\, \AtpE\right)\,\longrightarrow\, \Hun\left(X,\,{\Atlog}\right)$$
given by the inclusion $\AtpE\,\hookrightarrow\, {\Atlog}$ in \eqref{e7}
is the forgetful map that
sends any infinitesimal deformation of $\left(X,\, D,\, \Epar\right)$ to the
infinitesimal deformation of $\left(X,\, D,\, E\right)$ obtained by simply forgetting the parabolic data.

\subsection{Infinitesimal deformations of parabolic bundles with a subbundle}\label{sec:defparsubb}

Fix a pair $\left(X,\, D\right)$. As before, let $\Epar\,=\, \left(E,\, \{E^i_j\}, \,\{\alpha^i_j 
\}\right)$ be a parabolic vector bundle on $X$ with parabolic structure over $D$. Fix
a subbundle $0\, \not=\, F\, \subsetneq\, E$.

Let
$$
\EndpFEpar\, \subset\, \EndpEpar
$$
be the subsheaf that preserves $F$. The infinitesimal deformations of
the pair $\left(\Epar,\, F\right)$ (keeping the pair $\left(X,\, D\right)$ fixed) are parametrized by $\Hun\left(X,\,
\EndpFEpar\right)$. The homomorphism
$$
\Hun\left(X,\,\EndpFEpar\right)\, \longrightarrow\, \Hun\left(X,\,\EndpEpar\right)\, ,
$$
given by the inclusion of $\EndpFEpar$ in $\EndpEpar$, corresponds to
the forgetful map of infinitesimal deformations that forgets the subbundle $F$; recall
that $\Hun\left(X,\,\EndpEpar\right)$ is the space of infinitesimal deformations of
$\Epar$. The kernel of this forgetful
homomorphism corresponds to infinitesimal deformations of $F$ keeping $\Epar$ fixed.

Let
\begin{equation}\label{f2}
\mathrm{At}^F_p(E)\, \subset\, {\Atlog}
\end{equation}
be the coherent subsheaf whose sections over any open subset $U\, \subset\, X$
are all holomorphic differential operators
$$
D_U\, :\, E\vert_U \, \longrightarrow\, E\vert_U
$$
satisfying the following two conditions:
\begin{itemize}
\item for every holomorphic section $s$ of $E\vert_U$, and
every $x_j\, \in\, U$, if $s\left(x_j\right)\, \in\, E^i_j\, \subset\, E_{x_j}$, then
$D_U\left(s\right)\left(x_j\right)\, \in\, E^i_j$, and

\item $D_U\left(s\right)$ is a section of $F\vert_U$ if $s$ is a holomorphic section of $F\vert_U$.
\end{itemize}
Therefore, we actually have
\begin{equation}\label{f3}
\mathrm{At}^F_p(E)\, \subset\, \AtpE\, .
\end{equation}
We have the following short exact sequence of vector bundles on $X$:
\begin{equation}\label{z1}
0\, \longrightarrow\, \EndpFEpar\,\longrightarrow\, \mathrm{At}^F_p(E)
\, \longrightarrow\,\TXD \, \longrightarrow\, 0\, .
\end{equation}

Lemma \ref{lem1} has the following straightforward generalization:

\begin{lemma}\label{lem3}
The infinitesimal deformations of the quadruple $\left(X,\, D,\, \Epar,\, F\right)$ with
parabolic data of fixed type are parametrized by
$\Hun\left(X,\, {\rm At}^F_p(E)\right)$. The homomorphism
$$
\Hun\left(X,\, {\rm At}^F_p(E)\right)\, \longrightarrow\, \Hun\left(X,\, \AtpE\right)
$$
given by the inclusion $\mathrm{At}^F_p(E)\, \hookrightarrow\, \AtpE$ in \eqref{f3}
corresponds to the forgetful homomorphism that forgets $F$.
\end{lemma}

We note that the homomorphism
$$
\Hun\left(X,\, {\rm At}^F_p(E)\right)\, \longrightarrow\, \Hun\left(X,\, {\Atlog}\right)
$$
given by the inclusion ${\rm At}^F_p(E)\, \hookrightarrow\, {\Atlog}$ in \eqref{f2}
corresponds to the forgetful homomorphism that forgets $F$ as well as the parabolic structure on $E$
(recall that the infinitesimal deformations of the triple $\left(X,\, D,\, E\right)$
are parametrized by $\Hun\left(X,\, \Atlog\right)$).

\section{Isomonodromic deformations}

We will now recall the universal isomonodromic deformation of a given initial logarithmic connection, and how it encodes the 
infinitesimal deformation at the initial parameter of the underlying parabolic vector bundle.

\subsection{The initial connection}
Take $\left(X,\, D\right)$ as before.
As in Section \ref{se2.3}, let $E$ be a holomorphic vector bundle over $X$ of rank $r$, and let
$$\delta\, :\, \TXD \, \longrightarrow\, \Atlog$$ be a logarithmic connection on $E$,
singular over $D$. Let $\Epar$ be the parabolic vector bundle defined by the parabolic structure on $E$ given
by the residues of the logarithmic connection $\delta$ (see Section \ref{se2.3}).

\begin{lemma}\label{le-1}
The image $\delta \left(\TXD\right) \,\subset\, \Atlog$ is contained in the subsheaf
$\AtpE\,\subset\, {\Atlog}$ in \eqref{e7}.
\end{lemma}

\begin{proof}
Take any point $x_j\, \in\, D$. From the construction of the parabolic structure using
$\text{Res}\left(\delta\right)\left(x_j\right)$ it follows that $\text{Res}\left(\delta\right)\left(x_j\right)$ preserves the
quasiparabolic filtration of $\Epar$ over $x_j$. This means that
$$
\text{Res}\left(\delta\right)\left(x_j\right)\, \in\, \text{End}_{p,j}(E)\, \subset\, \EndE_{x_j}\, .
$$
{}From the definition of residue, given in \eqref{e5}, it now follows that
$\delta \left(\TXD\right) \,\subset\,\AtpE$.
\end{proof}

\subsection{The universal isomonodromic deformation}\label{sec:univdef}

For $\left(X,\, D\right)$ as before, fix an ordering of the points of $D$.
Let ${\mathcal T}_{g,n}$ be the Teichm\"uller space for
$\left(X,\, D\right)$. We briefly recall its construction, details can be found for example in \cite{Hubbard}. Let ${\mathbf C}_{g,n}$ denote the
space of all complex structures on $X$, and let $\text{Diff}$ denote the group of
all diffeomorphisms of $X$ that fix $D$ pointwise. Let
$$\text{Diff}^0\, \subset\, \text{Diff}$$
be the connected component containing the identity element.
Then we have $${\mathcal T}_{g,n}\,=\,
{\mathbf C}_{g,n}/\text{Diff}^0\, .$$
This ${\mathcal T}_{g,n}$ is a contractible complex manifold of complex dimension
$3g-3+n$. Note that there is a base point
\begin{equation}\label{bp}
t_0\, \,\in\, {\mathcal T}_{g,n}
\end{equation}
defined by the given complex structure on $X$.

There is a universal $n$-pointed Riemann surface $\left({\mathcal X},\, \left(s_1,\, \cdots , s_n\right)\right)$
over ${\mathcal T}_{g,n}$. This means that
\begin{equation}\label{p}
p\, :\, {\mathcal X}\,\longrightarrow\, {\mathcal T}_{g,n}
\end{equation}
is a holomorphic family of Riemann surfaces such that any fiber $p^{-1}\left(t\right)$ is the
Riemann surface associated to $t$, and
$
s_i\, :\, {\mathcal T}_{g,n}\,\longrightarrow\, {\mathcal X}$, for $ \ 1\, \leq\, i\,\leq\, n\, ,
$
are disjoint sections of the projection $p$ in \eqref{p}. The $n$-pointed Riemann surface
$(p^{-1}\left(t\right),\, (s_1(t),\, \cdots,\, s_n(t)))$ is represented by the point $t\,\in\,
{\mathcal T}_{g,n}$. Moreover, if $t_0$ denotes the base point
in \eqref{bp}, we have the following identification of $n$-pointed Riemann surfaces:
$$\left(p^{-1}\left(t_0\right),\, \left(s_1\left(t_0\right),\, \cdots , s_n\left(t_0\right)\right)\right)=\left(X, \, \left(x_1,\, \cdots , x_n\right)\right)\,=\, \left(X,\, D\right)$$
(recall that we have fixed an ordering of the points of $D$). \\Since ${\mathcal T}_{g,n}$ is contractible, 
the inclusion map
\begin{equation}\label{heq}
X\setminus D\, \hookrightarrow\, {\mathcal X}\setminus
\mathcal{D}
\end{equation}
as the fiber over $t_0$, where $\mathcal{D}\,:=\,
\left(\sqcup_{i=1}^n s_i\left({\mathcal T}_{g,n}\right)\right)$, is a homotopy equivalence.

As in Section \ref{se2.3}, let $E$ be a holomorphic vector bundle on $X$ or rank $r$, and let
\begin{equation}\label{id}
\delta\, :\, \TXD \, \longrightarrow\, \Atlog
\end{equation}
be a logarithmic connection on $E$, singular over $D$. There
exists a vector bundle ${\mathcal E}$ on ${\mathcal X}$, endowed with a flat logarithmic connection $\widetilde{\delta}$, singular over $\mathcal{D}$, such that 
 the restriction of $\left({\mathcal E},\, \widetilde{\delta}\right)$ to $p^{-1}\left(t_0\right)\,=\, X$ is
identified with $\left(E,\, \delta\right)$, where $t_0$ is the base point in \eqref{bp}. Let us briefly recall the construction (see \cite[Section 3]{Heu1} for details). 

Let
$$\rho\, :\, \pi_1\left(X\setminus D,\, x_0\right)\, \longrightarrow\, \text{GL}\left(E_{x_0}\right)$$
be the monodromy representation for the flat connection $\delta$; here $x_0\,\in\, X\setminus D$
is a fixed base point. Since the inclusion map in
\eqref{heq} is a homotopy equivalence, we have a homomorphism
$$
\pi_1\left({\mathcal X}\setminus
\mathcal{D},\, x_0\right)\,=\, \pi_1\left(X\setminus D,\, x_0\right)\,
\stackrel{\rho}{\longrightarrow}\, \text{GL}\left(E_{x_0}\right)
$$
which will be denoted by $\widetilde\rho$. This $\widetilde\rho$ produces a
holomorphic vector bundle $\widetilde{\mathcal E}$ over the complement ${\mathcal X}\setminus
\mathcal{D}$ equipped with a flat holomorphic
connection $ \widetilde{\delta}$ \cite{De}. Using an argument of Malgrange \cite{Mal} generalizing Deligne extensions in this context, this holomorphic vector bundle $\widetilde{\mathcal E}$
admits an extension ${\mathcal E}$ to ${\mathcal X}$ as a holomorphic vector bundle
such that
\begin{itemize}
\item the connection $\widetilde{\delta}$ extends to a logarithmic connection $\delta'$ on ${\mathcal E}$, and
\item the restriction of $\left({\mathcal E},\, \delta'\right)$ to $p^{-1}\left(t_0\right)\,=\, X$ is
identified with $\left(E,\, \delta\right)$, where $t_0$ is the point in \eqref{bp}.
\end{itemize}

The pair $\left({\mathcal E},\, \delta'\right)$ is unique and admits a universal property with respect to germs of isomonodromic
deformations of the same initial connection. It is therefore called the \emph{universal isomonodromic deformation} in \cite{Heu1}. In the current work, we will refer to the pair $\left({\mathcal E},\, \delta'\right)$ simply as 
\emph{the isomonodromic deformation} of the logarithmic connection $\left(E,\, \delta\right)$ on $X$.

For any $t\, \in\, {\mathcal T}_{g,n}$, the Riemann surface $p^{-1}\left(t\right)$ will be 
denoted by ${\mathcal X}_t$. The restriction of the holomorphic vector bundle $\mathcal E$ to 
${\mathcal X}_t$ will be denoted by ${\mathcal E}^t$. The restriction of the logarithmic 
connection $\delta'$ to ${\mathcal E}^t$ will be denoted by $\delta^t$.
 
\subsection{The underlying infinitesimal deformation of the parabolic bundle}\label{sec:parfamily}

We adopt the notation of Section \ref{sec:univdef}.
As shown in Section \ref{se2.3}, the logarithmic connection $\delta^t$ 
produces a parabolic structure on ${\mathcal E}^t$. The resulting parabolic vector bundle
on ${\mathcal X}_t$ will be denoted by ${\mathcal E}^t_\star$. Let
\begin{equation}\label{e10}
\cEpar\, \longrightarrow\, {\mathcal X}
\, \stackrel{p}{\longrightarrow}\,{\mathcal T}_{g,n}
\end{equation}
be the above family of parabolic vector bundles constructed from $\delta'$ (which in turn is
constructed from $\delta$).

\begin{lemma} Let $(\mathcal{E}, \delta)$ be the isomonodromic deformation of $(E,\delta)$. 
Then for each $1\, \leq\,
i\, \leq\, n$, the collection of parabolic weights and their multiplicities of ${\mathcal E}^t_\star$
at the parabolic point $s_i\left(t\right)\, \in\, {\mathcal X}_t$ is independent of $t$.
\end{lemma}

\begin{proof}For any $x_0^t$ in ${\mathcal X}_t\setminus \left(\sqcup_{i=1}^n s_i (t)\right)$ and any path from $x_0$ to $x_0^t$ in ${\mathcal X}\setminus
\mathcal{D}$, the holonomy of ${\delta}'$ yields an isomorphism $E_{x_0}\simeq {\mathcal E}^t_{x_0^t}$ identifying the monodromy $\rho$ of $\delta$ with the monodromy of $\delta^t$. Different choices of paths yield conjugated monodromy representations.
However, the conjugacy class of the local monodromy of $\delta^t$ around $s_i(t)$ does not depend on $t \in {\mathcal T}_{g,n}$. On the other hand,
the parabolic data at $s_i(t)$ is entirely encoded by the conjugacy class of the local monodromy of $\delta^t$ around $s_i(t)$. Indeed, the semisimple part of the local monodromy at $s_i(t)$ is conjugated to $\exp \left(\text{Res}^{ss}\left(\delta^t\right)\left(s_i(t)\right)\right)$ (see for example \cite[Theorem 1]{Bolibruch}). \end{proof}
 
In Lemma \ref{lem1} we saw that the
infinitesimal deformations of the triple $\left(X,\, D,\, \Epar\right)$, with parabolic
data of fixed type (fixed parabolic weights and their multiplicities), are parametrized by
$\Hun\left(X,\, \AtpE\right)$. On the other hand, for any $t\, \in\, {\mathcal T}_{g,n}$,
we have
$$
\mathrm{T}_t{\mathcal T}_{g,n}\,=\, \Hun\left({\mathcal X}_t,\, \mathrm{T}{\mathcal X}_t\otimes
{\mathcal O}_{{\mathcal X}_t}\left(-\sum_{j=1}^n s_j\left(t\right)\right)\right)\, .
$$
In particular, we have $\mathrm{T}_{t_0}{\mathcal T}_{g,n}\,=\,\Hun\left(X,\,\TXD\right)$, where $t_0$
is the base point in \eqref{bp}. Let
\begin{equation}\label{e9}
\gamma\, :\, \Hun\left(X,\,\TXD\right)\,=\, \mathrm{T}_{t_0}{\mathcal T}_{g,n}\, \longrightarrow\,
\Hun\left(X,\, \AtpE\right)
\end{equation}
be the classifying homomorphism corresponding to the family of parabolic
vector bundles in \eqref{e10} constructed from $\delta'$ (which in turn is 
constructed from $\delta$).

In Lemma \ref{le-1} we saw that $\delta \left(\TXD\right) \,\subset\,\AtpE$. Let
\begin{equation}\label{e9p}
\delta_*\, :\, \Hun\left(X,\,\TXD\right)\, \longrightarrow\, \Hun\left(X,\,\AtpE\right)
\end{equation}
be the homomorphism induced by $\delta\, :\, \TXD\, \longrightarrow\,
\AtpE$.

\begin{lemma}\label{lem2}
The homomorphism $\gamma$ in \eqref{e9} coincides with the homomorphism
$\delta_*$ in \eqref{e9p}.
\end{lemma}

\begin{proof} Lemma \ref{lem2} is straightforward to prove; the case without parabolic structure is dealt with in 
\cite[p.~131]{BHH1}. In the presence of parabolic structure it remains valid after appropriate modifications. 
\end{proof}

\section{The isomonodromic deformation contains stable parabolic bundles}

We are now ready to prove the first main result : if the initial connection is irreducible, 
the vector bundle corresponding to a generic fiber of the parameter space in
its (universal) isomonodromic deformation is parabolically stable.

\subsection{A criterion for extending a subbundle to the isomonodromy family}\label{sec:criterion}

Let $\delta$ be a logarithmic connection on $E$ as in \eqref{id}. Assume that $\delta$ is \emph{irreducible} in the sense
that no nonzero subbundle $E'\, \subsetneq\, E$ is preserved by $\delta$.

Let $F\, \subset\, E$ be a subbundle. We have the commutative diagram of sheaves on $X$:
\begin{equation}\label{cd1}\renewcommand{\labelstyle}{\textstyle}
\xymatrix{
& 0\ar[d] & 0\ar[d] \\
 0\ar[r]& \EndpFEpar\ar[d] \ar[r]& \mathrm{At}^F_p(E) \ar[d]\ar[r]&
\TXmD\ar@{=}[d] \ar[r] & 0\\
0 \ar[r] & \EndpE\ar[d]^{\gamma_0} \ar[r]^{\beta}& \AtpE\ar[d]^{\gamma_1} \ar[r]^{\sigma'} & \TXmD \ar[r]& 0\\
& \EndpE/\EndpFEpar\ar[d]\ar@{=}[r] & \EndpE/\EndpFEpar\ar[d]\\
& 0 & 0
 }
\end{equation}
 where the top short exact sequence is the one in \eqref{z1} and the short exact sequence
at the bottom is the one in \eqref{e8}. Consider the composition homomorphism
\begin{equation}\label{xi}
\TXmD \, \stackrel{\delta}{\longrightarrow}\, \AtpE
\, \stackrel{\gamma_1}{\longrightarrow}\, \EndpE/\EndpFEpar
\end{equation}
(Lemma \ref{le-1} says that the image of
$\delta$ is in $\AtpE$); this composition homomorphism will be
denoted by $f_0$. Since $F$ is not preserved by the connection $\delta$ by the
irreducibility assumption, we have
$$f_0\, \not=\, 0\, .$$ Let
\begin{equation}\label{cl}
{\mathcal L}\, \subset\, \EndpE/\EndpFEpar
\end{equation}
be the holomorphic line subbundle generated by the image $f_0\left(\TXmD\right)$. We note that
${\mathcal L}$ coincides with the inverse image, in $\EndpE/\EndpFEpar$,
of the torsion part
$$
(\left(\EndpE/\EndpFEpar\right)/f_0\left(\TXmD\right))_{\rm torsion}\, \subset\,
\left(\EndpE/\EndpFEpar\right)/f_0\left(\TXmD\right)$$ under the
quotient map $\EndpE/\EndpFEpar\,\longrightarrow\, \left(\EndpE/\EndpFEpar\right)/
f_0\left(\TXmD\right)$.

Now define
\begin{equation}\label{s1}
\text{End}^\delta_p\left(\Epar\right)\, :=\, \gamma^{-1}_0\left({\mathcal L}\right)
\, \subset\, \EndpE\ \ \text{ and } \ \
\mathrm{At}^\delta_p(E)\, :=\, \gamma^{-1}_1\left({\mathcal L}\right)\, \subset\,
\AtpE\, ,
\end{equation}
where $\mathcal L$ is the line subbundle in \eqref{cl}, and $\gamma_0,\, \gamma_1$
are the homomorphisms in \eqref{cd1}. Note that from \eqref{cd1} we have
the following commutative diagram of sheaves on $X$:
\begin{equation}\label{cd2}
\renewcommand{\labelstyle}{\textstyle}
\xymatrix{
& 0\ar[d] & 0\ar[d] \\
0 \ar[r] & \EndpFEpar \ar[d] \ar[r] & \mathrm{At}^F_p(E)\ar[d]^\mu \ar[r] & \TXmD\ar@{=}[d] \ar[r] & 0\\
0 \ar[r] & \text{End}^\delta_p\left(\Epar\right) \ar[d] \ar[r] &
\mathrm{At}^\delta_p(E)\ar[d]^{\gamma '} \ar[r]^{\sigma''} & \TXmD \ar[r] & 0\\
& {\mathcal L} \ar[d] \ar@{=}[r] & {\mathcal L}\ar[d] \\
 & 0 & 0
 }
\end{equation}
where $\sigma''$ and $\gamma'$ respectively are the restrictions of the homomorphisms
$\sigma'$ and $\gamma_1$ constructed in \eqref{cd1}.

{}From the definition of $\mathrm{At}^\delta_p(E)$ in \eqref{s1} it follows immediately that
the image of the connection homomorphism 
$\TXmD \, \stackrel{\delta}{\longrightarrow}\, \AtpE$ is contained
in the subbundle $\mathrm{At}^\delta_p(E)$.
Let
\begin{equation}\label{xi2}
\xi\, :\, \TXmD \, \longrightarrow\, \mathcal L
\end{equation}
be the homomorphism given by the composition $f_0$ in \eqref{xi}.

Consider the family of parabolic bundles 
$$
\cEpar\, \longrightarrow\, {\mathcal X}
\, \stackrel{p}{\longrightarrow}\,{\mathcal T}_{g,n}
$$
constructed in \eqref{e10} using $\delta'$ (which is constructed from $\delta$).
From the commutative diagram in \eqref{cd2} we can now deduce the following proposition. 
 
\begin{proposition}\label{prop1}
If the subbundle $F\, \subset\, E$ extends to a subbundle ${\mathcal F}$ of $ {\mathcal E}$ over the
first order infinitesimal neighborhood of the point $t_0\, \in\, \mathcal{Y}$, where $\mathcal{Y}$ is a closed analytic subset of ${\mathcal T}_{g,n}$,
then the homomorphism defined by the composition 
$$
\mathrm{T}_{t_0}{\mathcal Y}\, \hookrightarrow\, \mathrm{T}_{t_0}{\mathcal T}_{g,n}=\Hun\left(X,\, \TXmD\right)\, \stackrel{ \xi_*}{\longrightarrow}\, \Hun\left(X,\, {\mathcal L}\right)\, ,
$$
induced by $\xi$ in \eqref{xi2}, vanishes identically.
\end{proposition}

\begin{proof}
Assume that the subbundle $F\, \subset\, E$ extends
to the first order infinitesimal neighborhood of $t_0\, \in\, \mathcal{Y}\,\subset\,
{\mathcal T}_{g,n}$. Consequently, we have a classifying homomorphism
$$\mathrm{cl}_{(X,D,\Epar,F)} \,:\, \mathrm{T}_{t_0}{\mathcal Y}\,
\longrightarrow\, \Hun\left(X,\, {\rm At}^F_p(E)\right)$$ to the space of infinitesimal
deformations of $(X,\,D,\,\Epar,\,F)$ (that is of quadruples given by curve, punctures,
parabolic bundle and subbundle in the isomonodromic deformation). Denoting forgetful morphisms
simply by ``$\circ$'', and also adopting a similar notation for the other classifying maps, by
Lemma \ref{lem3} and lemma \ref{lem2}, the following diagram of homomorphisms is commutative:
$$\xymatrix{
&& \Hun\left(X,{\rm At}^F_p(E)\right) \ar[r]^{\mu_*} \ar[d]^{\circ}
& \Hun\left(X,{\rm At}^\delta_p(E)\right)\ar[r]^{\ \ \ \gamma'_*}\ar[dl]^{\circ} 
& \Hun\left(X,\mathcal{L}\right)\, .
\\
\mathrm{T}_{t_0}{\mathcal Y} \ar@{_{(}->}[d]\ar@/^2pc/[urr]^{\mathrm{cl}_{(X,D,\Epar,F)}}\ar@/^1pc/[drr]_{\mathrm{cl}_{(X,D)}}\ar@/^1pc/[rr]^{\mathrm{cl}_{(X,D,\Epar)}}
&& \Hun\left(X,\, {\rm At}_p(E)\right)\ar[d]^{\circ}
\\
\mathrm{T}_{t_0}{\mathcal T}_{g,n} \ar@{=}[rr]
&& \Hun\left(X,\, \TXD\right)\ar@/_3pc/[uurr]^{\xi_*}\ar@/_2pc/[uur]^{\delta_*}
}$$
(in the above diagram ``$\circ$'' denotes the homomorphisms of cohomologies induced by
the natural inclusions of coherent sheaves).
The result now simply follows from the fact that the top row is exact according to \eqref{cd2}.
\end{proof}

\begin{theorem}\label{thm1}
Let $X$ be a compact Riemann surface of genus $g\,\geq\, 2$, and let $D$ be a divisor on $X$. 
Let $\delta$ be an irreducible logarithmic connection, singular over $D$, on a holomorphic
vector bundle $E\,\longrightarrow\, X$. Consider the family of parabolic vector bundles
$$
\cEpar\, \longrightarrow\, {\mathcal X}
\, \stackrel{p}{\longrightarrow}\,{\mathcal T}_{g,n}$$
underlying the isomonodromic deformation of $\left(E,\, \delta\right)$ as in
Section \ref{sec:parfamily}, and denote, for any $t\,\in\, {\mathcal T}_{g,n}$, by $\cEpar^t$ the
corresponding parabolic
vector bundle over $\mathcal{X}_t\,=\,p^{-1}\left(t\right)$ with parabolic structure over
the divisor $\left(s_1(t),\, \cdots\, , s_n(t) \right)$.
Denote $$\begin{array}{rcl}{\mathcal Y}&:=&\{t\in {\mathcal T}_{g,n}~|~\cEpar^t \textrm{
is not parabolically semistable.} \}\vspace{.2cm}\\{\mathcal Y'}&:=&\{t\in {\mathcal T}_{g,n}~|~\cEpar^t \textrm{
is not parabolically stable.} \}\end{array}$$
Then ${\mathcal Y}$ and ${\mathcal Y'}$ are closed analytic subsets of ${\mathcal T}_{g,n}$, whose codimensions are bounded as follows: 
$$\mathrm{codim}\left({\mathcal Y}\right)\geq g\, ; \quad \mathrm{codim}\left({\mathcal Y}'\right)\geq g-1\, .$$
\end{theorem}

\begin{proof}
The mechanics of the proof of this theorem are identical to the proofs of Proposition 5.3 of \cite[p.~138]{BHH1} (concerning $\mathcal{Y}$) and Proposition 5.4 of \cite[p.~139]{BHH1} (concerning $\mathcal{Y}'$) up to some minor modifications. We will therefore be brief.
The fact that $\mathcal{Y}$ and $\mathcal{Y}'$ defined as in the statement are closed analytic subsets of ${\mathcal T}_{g,n}$ follows from \cite{Nitsure}. Indeed, one can write $ \mathcal{Y}'$ as a union of strata corresponding to types $k$ of nontrivial Harder-Narasimhan filtrations, and the results of \cite{Nitsure} tell us that 
the union of strata corresponding to types greater or equal to a fixed $k$ forms a closed subset. On the other hand, within the moduli space of semi-stable objects, stable ones form an open subset.
Let $0\neq F\subset E$ be a destabilizing subbundle, i.e., 
$$\frac{\textrm{par-deg}\left(\Fpar\right)}{\mathrm{rank}\left(\Fpar\right)}\,>\,
\frac{\textrm{par-deg}\left(\Epar\right)}{\mathrm{rank}\left(\Epar\right)}\, ,
\quad (\textrm{respectively,} \quad \frac{\textrm{par-deg}\left(\Fpar\right)}{\mathrm{rank}
\left(\Fpar\right)}\,\geq\, \frac{\textrm{par-deg}\left(\Epar\right)}{\mathrm{rank}\left(\Epar\right)}).$$
Then, as is Section \ref{sec:criterion}, we have a short exact sequence of sheaves on $X$
\begin{equation}\label{eqtor}0\longrightarrow \TXmD\stackrel{\xi}{\longrightarrow}\mathcal{L}
\longrightarrow T^\delta\longrightarrow 0\, ,\end{equation} where $T^\delta$ is a torsion sheaf because $\xi\neq 0$ by
irreducibility of $\delta$. 

We will show that
\begin{equation}\label{ze1}
\mathrm{degree} (\mathcal{L} )<0\, , \quad
(\textrm{respectively,} \quad \mathrm{degree} (\mathcal{L} )\leq 0)
\end{equation}
in the stable (respectively, semistable) case.

For this first consider the Harder-Narasimhan filtration of the parabolic endomorphism bundle ${\rm End}(\Epar)\,=\,
\Epar\otimes E^*_\star$. Let $W_\star\, \subset\, {\rm End}(\Epar)$ is the part of this filtration for nonnegative parabolic
weights. Then all the successive quotients of the Harder-Narasimhan filtration of the quotient parabolic bundle
${\rm End}(\Epar)/W_\star$ have negative parabolic degree. On the other hand, when $\Epar$ is parabolic
semistable, for the socle filtration of ${\rm End}(\Epar)$, all the successive quotients of the filtration have
parabolic degree to be zero. In the stable case, $\mathcal{L}$ is a subsheaf of the quotient parabolic bundle
${\rm End}(\Epar)/W_\star$, and hence the parabolic degree of $\mathcal{L}$ with the induced parabolic structure is negative.
This implies that the degree of $\mathcal{L}$ is negative. In the semistable case, $\mathcal{L}$ is a subsheaf of the quotient
of the socle filtration, so the parabolic degree of $\mathcal{L}$ with the induced parabolic structure is nonpositive. Hence
the degree of $\mathcal{L}$ is nonpositive in this case.
Also form the result on p. 705 of \cite{AAB} it follows that that the degree of $\mathcal{L}$ must be negative,
in the stable case, and negative or zero, in the semi-stable case. 

From the long exact sequence associated to the short exact sequence \eqref{eqtor}, one then deduces 
\begin{equation}\label{eqdim}
\mathrm{dim}\left(\xi_*\Hun\left(X,\TXmD\right)\right)\geq g\, , \quad (\textrm{respectively,} \quad 
\mathrm{dim}\left(\xi_*\Hun\left(X,\TXmD\right)\right)\geq g-1)\, .
\end{equation}
Up to replacing $t_0$ by a generic element of 
$\mathcal{Y}$ respectively, $\mathcal{Y}'$, we may assume that in the infinitesimal neighborhood of $t_0$ in $\mathcal{Y}$ 
respectively, $\mathcal{Y}'$, the destabilizing subbundle $F$, which we take to be maximal, in the Harder-Narasimhan sense, 
extends; this follows from the picture of $\mathcal{Y}$ respectively, $\mathcal{Y}'$ as a union of strata. Then Proposition 
\ref{prop1}, in combination with \eqref{eqdim}, yields the desired estimate for the codimension. \end{proof}
 
\section{Infinitesimal deformations of parabolic Higgs bundles}

This section is dedicated to prove our second main result: in the rank two case, if the initial 
connection is irreducible, the vector bundle corresponding to a generic fiber of the parameter 
space in the (universal) isomonodromic deformation is parabolically very stable. We shall 
proceed in a way similar to what lead to the first main result. Namely, after recalling the 
basic definitions, we will establish the deformation theory of parabolic Higgs bundles over 
varying base curves, as well as the obstruction space of deformations of non-zero nilpotent 
Higgs fields. These results will then be applied to the isomonodromic deformation.

Let $\left(X,\,D\right)$ be as before a compact Riemann surface of genus $g\,\geq\, 2$ endowed with
$n$ ordered marked points. 
Let $\Epar $ be a vector bundle $E\,\longrightarrow\, X$ endowed with a parabolic structure over $D$ as before. However, from
now on we will always assume that 
$$\text{rank}(E)\, =\, 2\, .$$

For each $x_j\in D$, the parabolic filtration of $E_{x_j}$ in \eqref{qf} then is of length $n_j\leq 2$.

\subsection{Very stable parabolic Higgs bundles}\label{sec:verystable}

Let us recall the notion of $\Epar$ being parabolically very stable. 

Consider the vector bundle 
 $\EndpE$ in \eqref{eq:defEndpE} and define
\begin{equation}\label{ctb}
\EndzepE\, \subset\, \EndpE
\end{equation}
to be the coherent subsheaf defined by the endomorphisms that are nilpotent 
with respect to the quasiparabolic filtration over every point of $D$, \emph{i.e.,} for a
section $s$ of $\EndzepE$, we have
$$
s\left(E^i_j\right)\, \subseteq\, E^{i+1}_j
$$
for all $x_j\, \in\, D$ in the domain of definition of $s$ and all $1\,\leq\, i\,\leq\, n_j$ (as in \eqref{qf}).

\begin{remark}
Since $\EndE^\vee\, =\, \EndE$ with the isomorphism given by the bilinear pairing defined by $A\otimes B\, \longmapsto\,
\text{trace}(AB)$, we have a fiberwise nondegenerate pairing
$$
(\EndE\otimes {\mathcal O}_X(D))\otimes (\EndE\otimes {\mathcal O}_X(D))\, \longrightarrow\,
{\mathcal O}_X(2D)
$$
given by trace. For this pairing, the image of $\EndpE\otimes (\EndzepE\otimes \mathcal{O}_X(D))$ is evidently
contained in ${\mathcal O}_X\, \subset\, {\mathcal O}_X(2D)$. It is now straightforward to check that this
restricted pairing produces an isomorphism
\begin{equation}\label{n1}
\EndpE^\vee\,=\, \EndzepE\otimes \mathcal{O}_X(D)\, .
\end{equation}
\end{remark}

A \textit{Higgs field} on a parabolic vector bundle $\Epar$ is a holomorphic section of
$\EndzepE\otimes \KX\otimes \mathcal{O}_X(D)$, where $\EndzepE$ is the
vector bundle constructed in \eqref{ctb}. A \textit{Higgs bundle} is a pair $\left(\Epar,\, \theta\right)$, where $\Epar$
is a parabolic vector bundle and $\theta$ is a Higgs field on $\Epar$. The Higgs field $$
\theta\, \in\, \Hze\left(X,\, \EndzepE\otimes \KX\otimes \mathcal{O}_X(D)\right)
$$ is
called \emph{nilpotent} if $\theta^2\,=\,0$. 
A parabolic Higgs bundle $\left(\Epar,\, \theta\right)$ is called {nilpotent} if $\theta$ is nilpotent.

A parabolic vector bundle $\Epar$ is called \textit{parabolically very stable} if it does not admit any nonzero
nilpotent Higgs field. It can be proved that a parabolically very stable vector bundle $\Epar$ is automatically parabolically stable.
To prove this, assume that $\Epar$ is not stable. then there is a line subbundle $L\, \subset\, E$ such that
\begin{equation}\label{inq}
\text{par-deg}\left(\Lpar\right)\,\geq\, \frac{\text{par-deg}\left(\Epar\right)}{2}\, ,
\end{equation}
where $\Lpar$ is the parabolic line bundle given by the parabolic structure on $L$ induced by the
parabolic structure on $\Epar$. 
Denote $D'\,:=\, \{x_j\in D~|~E_j^2\neq \{0\}\} $ and 
\begin{equation}\label{eqDL} D_L\,:=\,\{x_j\,\in \, D'~\mid ~L_{x_j}\,=\, E_{j}^2\}\, .\end{equation}

From \eqref{inq} it follows that 
$$\text{degree}\left(\text{Hom}\left(E/L\, ,\, L\right)\otimes \mathcal{O}_X(D_L)\right) \geq \text{degree}(D_L)+\sum_{D'-D_L}(\alpha_j^2 -\alpha_j^1) -\sum_{D_L}(\alpha_j^2 -\alpha_j^1)\geq 0\, .$$
Then the line bundle $\text{Hom}\left(E/L\, ,\, L\right)\otimes \KX\otimes \mathcal{O}_X(D')$ has a non-zero holomorphic section
by Riemann--Roch theorem. A nonzero holomorphic section $\zeta$ of $\text{Hom}\left(E/L\, ,\, L\right)\otimes \KX\otimes \mathcal{O}_X(D')$
defines a nonzero nilpotent Higgs field on $\Epar$ using the composition
$$
E\,\longrightarrow\, E/L\,\stackrel{\zeta}{\longrightarrow}\, L\otimes \KX\otimes \mathcal{O}_X(D')\,\longrightarrow\, E
\otimes \KX\otimes \mathcal{O}_X(D')\, ,
$$
where $\,\longrightarrow\, E/L$ is the quotient map; the other homomorphism
$L\otimes \KX\otimes \mathcal{O}_X(D')\,\longrightarrow\, E
\otimes \KX\otimes \mathcal{O}_X(D')$ is the tensor product of the inclusion $L\, \hookrightarrow\, E$ with the
identity map of $\KX\otimes \mathcal{O}_X(D')$. Therefore, $\Epar$ is not parabolically very stable.

Note that the kernel of the above composition homomorphism is precisely $L$. 

\subsection{Infinitesimal deformations of a parabolic Higgs bundle on a fixed curve}\label{sec:Higgsfixedcurve}

Let $\left(\Epar,\, \theta\right)$ be a parabolic Higgs bundle of rank 2 over a fixed pointed curve $\left(X,D\right)$. As recalled 
in Section \ref{sec:fixed curve}, the infinitesimal deformations of $\Epar$ are parametrized by $\Hun\left(X,\, \EndpE\right)$. 
These of course need to be reflected in the infinitesimal deformations of the pair $\left(\Epar,\, \theta\right)$. Using Serre 
duality, and \eqref{n1}, the dual of the space of infinitesimal deformations of $\Epar$ is $\Hze\left(X,\, \EndzepE\otimes 
\OXD\otimes \KX\right)$, where $\KX$ is the holomorphic cotangent bundle of $X$. As shown in \cite{BR}, this dual space corresponds 
to the infinitesimal deformations of Higgs fields $\theta$ on a fixed parabolic bundle $\Epar$. Let us recall how these two 
infinitesimal deformation spaces fit together to construct the infinitesimal deformation space of pairs $\left(\Epar,\, 
\theta\right)$.

Let
\begin{equation}\label{ft}
f_\theta\, :\, \EndpE\, \longrightarrow\, \EndzepE
\otimes \KX\otimes \OXD
\end{equation}
be the homomorphism defined by $A\, \longmapsto\, \theta\circ A- A\circ\theta$. Now we have a two-term complex
${\mathcal C}^{\left(\Epar,\theta\right)}_{\bullet}$ of sheaves on $X$
$$
{\mathcal C}^{\left(\Epar,\theta\right)}_0\,:=\, \EndpE \, \stackrel{f_\theta}{\longrightarrow}\,
{\mathcal C}^{\left(\Epar,\theta\right)}_1\,:=\, \EndzepE \otimes \KX\otimes \OXD\, .
$$
The infinitesimal deformations of $\left(\Epar,\, \theta\right)$, keeping $\left(X,\, D\right)$ fixed, are parametrized by
the hypercohomology ${\mathbb H}^1\hskip-3pt \left({\mathcal C}^{\left(\Epar,\theta\right)}_{\bullet}\right)$ \cite{BR}. Consider the following short
exact sequence of complexes.
$$\renewcommand{\labelstyle}{\textstyle}
\xymatrix{
0\ar[d] & & 0\ar[d]\\
 0\ar[d]\ar[rr] & & \EndzepE\otimes \KX\otimes \OXD\ar@{=}[d]\\ 
 \EndpE\ar@{=}[d]\ar[rr]^{f_\theta \ \ \ \ \ }&& \EndzepE
\otimes \KX\otimes \OXD\ar[d]\\
 \EndpE\ar[d]\ar[rr]& & 0\ar[d]\\
 0 & & 0
 }
$$
It produces an exact sequence of hypercohomologies
$$
\Hze\left(X,\, \EndzepE\otimes \KX\otimes \OXD\right) \, \stackrel{a}{\longrightarrow}\,
{\mathbb H}^1\hskip-3pt \left({\mathcal C}^{\left(\Epar,\theta\right)}_{\bullet}\right) \, \stackrel{b}{\longrightarrow}\,
\Hun\left(X,\, \EndpE\right)\, .
$$
The above homomorphism $a$ corresponds to changing the Higgs field keeping $\Epar$ fixed, and $b$
corresponds to the forgetful map that sends an infinitesimal deformation of $\left(\Epar,\, \theta\right)$ to the corresponding
infinitesimal deformation of $\Epar$ by simply forgetting $\theta$.

\subsection{Infinitesimal deformations of a parabolic Higgs bundle on moving curve}

In Section \ref{sec:Higgsfixedcurve}, we recalled the infinitesimal deformation space of parabolic Higgs fields with fixed pointed base curve. On the other hand, in Section \ref{sec:varbase}, we stated that the infinitesimal deformation space of the triple $\left(X,D,\Epar\right)$ is given by $\Hun\left(X,\AtpE\right)$. We shall now explain how these two spaces fit together to form the infinitesimal deformation space of the quadruple $\left(X,D,\Epar, \theta\right)$.

There is a natural homomorphism
\begin{equation}\label{eta}
\eta\, :\, \AtpE\, \longrightarrow\, \text{Diff}^1_X\left(\EndzepE\otimes
\KX\otimes \OXD,\, \EndzepE\otimes \KX\otimes \OXD\right)\, ,
\end{equation}
where $\AtpE$ is constructed in \eqref{e7}. To construct $\eta$, consider the homomorphism
$$
\Atlog\, \longrightarrow\, \text{Diff}^1_X\left(\EndE\otimes
\KX,\, \EndE\otimes \KX\right)
$$
constructed in \cite[p.~635, (4.1)]{BHH2}, where $\Atlog$ is constructed in \eqref{e2}; in essence, one combines the action on sections of $E, End(E)$ with a Lie derivative on $K$ (but see \cite{BHH2}) . It is
straight-forward to check that this homomorphism produces a homomorphism as in \eqref{eta} (see
Section 4.1 of \cite{BHH2}).
We have the homomorphism
\begin{equation}\label{eta2}
\eta_\theta\, :\, \AtpE\, \longrightarrow\,\EndzepE \otimes \KX\otimes \OXD
\, ,\ \ s \, \longmapsto\, \eta\left(s\right)\left(\theta\right)\, .
\end{equation}

Denote the quadruple $\left(X,\, D,\, \Epar,\, \theta\right)$ by $\underline{z}$.
Let ${\mathcal A}^{\underline{z}}_{\bullet}$ be the following two-term complex
of sheaves on $X$:
$$
{\mathcal A}^{\underline{z}}_0\,:=\, \AtpE \, \stackrel{\eta_\theta}{\longrightarrow}\,
{\mathcal A}^{\underline{z}}_1\,:=\, \EndzepE \otimes \KX\otimes \OXD\, ,
$$
where $\eta_\theta$ is the homomorphism in \eqref{eta2}.
The infinitesimal deformations of $\underline{z}\,=\, \left(X,\, D,\, \Epar,\, \theta\right)$ are parametrized by
the hypercohomology ${\mathbb H}^1\hskip-3pt \left({\mathcal A}^{\underline{z}}_{\bullet}\right)$. Consider the following short
exact sequence of complexes.
$$\renewcommand{\labelstyle}{\textstyle}
\xymatrix{
0 \ar[d] && 0\ar[d] \\
0 \ar[d] \ar[rr] && \EndzepE\otimes \KX\otimes \OXD\ar@{=}[d]\\
 \AtpE \ar@{=}[d] \ar[rr]^{\eta_\theta \ \ \ \ \ \ \ } && \EndzepE
\otimes \KX\otimes \OXD\ar[d]\\
\AtpE \ar[d] \ar[rr] && 0\ar[d]\\
 0 && 0
 }
$$
 It produces an exact sequence of hypercohomologies
$$
\Hze\left(X,\, \EndzepE\otimes \KX\otimes \OXD\right) \, \stackrel{a'}{\longrightarrow}\,
{\mathbb H}^1\hskip-3pt \left({\mathcal A}^{\underline{z}}_{\bullet}\right) \, \stackrel{b'}{\longrightarrow}\,
\Hun\left(X,\, \AtpE\right) \, .
$$
The above homomorphism $a'$ corresponds to changing the Higgs field keeping the triple $\left(X,\, D,\, \Epar\right)$ fixed, and
$b'$ corresponds to the forgetful map that sends an infinitesimal deformation of $\left(X,\, D,\, \Epar,\, \theta\right)$ to the corresponding
infinitesimal deformation of $\left(X,\, D,\, \Epar\right)$ by simply forgetting $\theta$; recall from Lemma \ref{lem1} that
$\Hun\left(X,\, \AtpE\right)$ parametrizes the infinitesimal deformations of $\left(X,\, D,\, \Epar\right)$.

\subsection{Infinitesimal deformations of a nilpotent parabolic Higgs bundle} We shall now construct the obstruction space, \emph{i.e.} when the infinitesimal deformation of a nonzero nilpotent parabolic Higgs field remains nilpotent. 

Let $\left(\Epar,\, \theta\right)$ be a parabolic Higgs bundle of rank 2 over a fixed pointed curve $\left(X,D\right)$ as before. 
Now assume that the Higgs field $\theta$ on $\Epar$ is nonzero nilpotent. Let
$$
L\, :=\, \text{kernel}\left(\theta\right) \, \subset\, E
$$
be the corresponding holomorphic line subbundle and denote $Q:=E/L$ the quotient bundle.
From the exact sequence $$0\,\longrightarrow\, L\,\longrightarrow\, E
\,\longrightarrow\, Q\,\longrightarrow\, 0$$ and its dual sequence, we obtain an exact sequence
$$0\, \longrightarrow\, \mathrm{End}_n^L(E):=Q^\vee\otimes L\, \longrightarrow\, E^\vee\otimes L\oplus Q^\vee\otimes E\, \longrightarrow\,\EndE\, \longrightarrow\, L^\vee\otimes Q\, \longrightarrow\,0\, , $$
factoring through 
$$
\text{End}^L(E)\, :=\,\{s\,\in \,\EndE ~\mid ~s(L)\, \subset\, L\}
 $$
such that we have the following two short exact sequences:
$$\xymatrix{0\ar[r]&\mathrm{End}_n^L(E)\ar[r] &E^\vee\otimes L \oplus Q^\vee\otimes E\ar[r]&\text{End}^L(E) \ar[r]& \,0\ \\
& 
&0\ar[r]&\text{End}^L(E)\ar@{=}[u]\ar[r]&\EndE\ar[r]&L^\vee\otimes Q\ar[r]&0\ . }$$

We note that $\text{rank}\left(\text{End}^L(E)\right)\,=\, 3$, and $\text{rank}\left(\text{End}_n^L(E)\right)\,=\, 1$. The line bundle
$\mathrm{End}_n^L(E)=\mathrm{Hom}(Q,L)$ defined above corresponds to those endomorphisms of $E$ which respect the filtration $0\subset L\subset E$ and which are moreover nilpotent; it is also the kernel of the natural projection $\text{End}^L(E)
\,\longrightarrow\, \text{End}(L)\oplus \text{End}(Q)$. 
Now define $\EndpLEpar\,:=\, \text{End}^L(E)\cap \EndpE\, \subset\,
\EndE$ as in Section \ref{sec:defparsubb} and set
$$
\EndnLEpar\,:=\, \text{End}_n^L(E)\cap \EndpE\, \subset\,
\EndE\, .
$$
We have the following two term complex ${\mathcal D}^{\left(\Epar,\theta\right)}_{\bullet}$ of sheaves on $X$:
$$
{\mathcal D}^{\left(\Epar,\theta\right)}_0\, =\, \EndpLEpar\, \stackrel{f'_\theta}{\longrightarrow}\,
{\mathcal D}^{\left(\Epar,\theta\right)}_1\, =\, \EndnLEpar\otimes \KX\otimes \OXD\, ,
$$
where $f'_\theta$ is the restriction of the homomorphism $f_\theta$ in \eqref{ft}.
The infinitesimal deformations of $\left(\Epar,\, \theta\right)$ in the moduli of nilpotent parabolic Higgs bundles
(keeping $\left(X,\, D\right)$ fixed) are parametrized by ${\mathbb H}^1\hskip-3pt \left({\mathcal D}^{\left(\Epar,\theta\right)}_{\bullet}\right)$
\cite{BR}.

Let
$ 
\AtthpE\, \subset\, \AtpE
$ be as in Section \ref{sec:defparsubb} (with $F=L$). 
 The homomorphism $\eta_\theta$ in
\eqref{eta2} maps $\AtthpE$ to $\text{End}^n_L\left(\Epar\right)\otimes \KX\otimes \OXD$.
As before, denote the quadruple $\left(X,\, D,\, \Epar,\, \theta\right)$ by $\underline{z}$.
We have the following two term complex ${\mathcal B}^{\underline{z}}_{\bullet}$ of sheaves on $X$:
$$
{\mathcal B}^{\underline{z}}_0\, =\, \AtthpE\, \stackrel{\eta'_\theta}{\longrightarrow}\,
{\mathcal B}^{\underline{z}}_1\, =\, \EndnLEpar \otimes \KX\otimes \OXD\, ,
$$
where $\eta'_\theta$ is the restriction of the homomorphism $\eta_\theta$ in \eqref{eta2}.

The infinitesimal deformations of $\underline{z}\,=\, \left(X,\, D,\, \Epar,\, \theta\right)$ in the moduli of nilpotent
parabolic Higgs bundles are parametrized by ${\mathbb H}^1\hskip-3pt \left({\mathcal B}^{\underline{z}}_{\bullet}\right)$.
The morphism ${\mathbb H}^1\hskip-3pt \left({\mathcal B}^{\underline{z}}_{\bullet}\right)
\,\longrightarrow\, {\mathbb H}^1\hskip-3pt \left({\mathcal A}^{\underline{z}}_{\bullet}\right)$ forgetting that the Higgs field remains nilpotent along the infinitesimal deformation is obtained from the morphism of complexes
$$
\xymatrix{{\mathcal B}^{\underline{z}}_0\, =\, \AtthpE\, \ar[r]^{\eta'_\theta\hskip40pt \ }\ar[d]&
{\mathcal B}^{\underline{z}}_1\, =\, \EndnLEpar \otimes \KX\otimes \OXD \ar[d]\, \\
{\mathcal A}^{\underline{z}}_0\, =\, \AtpE\, \ar[r]^{\eta'_\theta\hskip40pt \ }&
{\mathcal A}^{\underline{z}}_1\, =\, \EndzepE \otimes \KX\otimes \OXD\, 
}
$$ induced by the identity. The morphism ${\mathbb H}^1\hskip-3pt \left({\mathcal B}^{\underline{z}}_{\bullet}\right)\,\longrightarrow\,\Hun \left(X,{\rm At}^L_p(E)\right)$ however,
which to a infinitesimal deformation of $\left(X,\, D,\, \Epar,\, \theta\right)$ with nilpotent Higgs field associates the underlying
infinitesimal deformation of $\left(X,\, D,\, \Epar,\, L\right)$ with $L\,=\,\mathrm{kernel}(\theta)$ is obtained from the natural morphism of complexes 

$$
\xymatrix{{\mathcal B}^{\underline{z}}_0\, =\, \AtthpE\, \ar[r]^{\eta'_\theta\hskip40pt \ }\ar[d]&
{\mathcal B}^{\underline{z}}_1\, =\, \EndnLEpar \otimes \KX\otimes \OXD \ar[d]\, \\
 \AtthpE\, \ar[r] &
0 \, ;
}
$$
note that the first hypercohomology space of the complex below coincides with $ \Hun\left(X,{\rm At}^L_p(E)\right)$.

\subsection{The isomonodromic deformation contains very stable parabolic bundles}

We have now established the necessary ingredients of our second main result: 

\begin{theorem}\label{thm2}
Let $X$ be a Riemann surface of genus $g\geq 2$ and let $D$ be a divisor on $X$. 
Let $\delta$ be an irreducible logarithmic connection, singular over $D$, on a rank 2 vector bundle $E\,\longrightarrow\, X$.
Consider the family of parabolic bundles
$$
\cEpar\, \longrightarrow\, {\mathcal X}
\, \stackrel{p}{\longrightarrow}\,{\mathcal T}_{g,n}\, 
$$ underlying the universal isomonodromic deformation of $\left(E, \delta\right)$ as in Section \ref{sec:parfamily} and denote, for any $t\in {\mathcal T}_{g,n}$, by $\cEpar^t$ the corresponding parabolic
vector bundle over $\mathcal{X}_t=p^{-1}\left(t\right)$ with parabolic structure over $\left(s_1(t),\, \cdots , s_n(t)\right)$.
Denote $$\begin{array}{rcl}{\mathcal Y''}&:=&\{t\in {\mathcal T}_{g,n}~|~\cEpar^t \textrm{
~is ~not ~parabolically ~very ~stable.} \} \end{array}$$
Then ${\mathcal Y''}$ is a proper closed analytic subset of ${\mathcal T}_{g,n}$.
\end{theorem}

\begin{proof}
The proof of this theorem is identical to the proof of Theorem 5.2 of \cite[p.~639]{BHH2} after
some minor modifications. We will therefore be brief. Let $\theta$ be a nonzero nilpotent Higgs bundle on the parabolic vector bundle $E_\star$ corresponding to the initial parameter of the isomonodromic deformation. Denote $L:=\mathrm{kernel}(\theta)$ as before and let $D_L$ be as in equation \eqref{eqDL}. Recall the commutative diagram \eqref{cd1} with exact rows and columns: \begin{equation}\label{cdHiggs}
\renewcommand{\labelstyle}{\textstyle}
\xymatrix{
& 0\ar[d] & 0\ar[d] \\
0 \ar[r] & \EndpLEpar \ar[d] \ar[r] & \mathrm{At}^L_p(E)\ar[d]^{\mu_1} \ar[r] & \TXmD\ar@{=}[d] \ar[r] & 0\\
0 \ar[r] & \text{End}_p\left(\Epar\right) \ar[d] \ar[r] &
\mathrm{At}_p(E)\ar[d]^{\gamma_1} \ar[r]^{\sigma'} & \TXmD \ar[r]\ar[dl]^q & 0\\
& \mathcal{Q} \ar[d] \ar@{=}[r] & \mathcal{Q} \ar[d] \\
 & 0 & 0\, .
 }
\end{equation}
Here $$\mathcal{Q}:=L^\vee \otimes Q\otimes \mathcal{O}_X(-D_L)$$ with the notation of \eqref{eqDL}, 
 and $q:=\gamma_1\circ \delta$. 
 Since $\delta$ is irreducible, we have $q\neq 0$. Since $\mathcal{Q}$ is a line bundle, we obtain an exact sequence
$$0\, \longrightarrow\, \TXD \, \stackrel{q}{\longrightarrow}\, \mathcal{Q} \, \longrightarrow\, \mathbb{T}\, \longrightarrow\,0\, , $$
where $\mathbb{T}$ is a torsion sheaf. From the corresponding long exact sequence, we have that the induced 
morphism $$q_* \, :\, \Hun \left(X,\TXD \right) \, \longrightarrow \,\Hun\left(X,\mathcal{Q}\right)$$ of 
cohomology spaces is surjective. Since $\theta$ is nonzero nilpotent with kernel $L$, it induces a non-zero 
section of $\mathcal{Q}^\vee\otimes \KX = \mathrm{Hom}(Q,L)\otimes \mathcal{O}_X(D_L)\otimes \KX \subset 
\mathrm{Hom}(Q,L)\otimes \mathcal{O}_X(D')\otimes \KX$. In particular, using Serre duality, we have 
$\Hun\left(X,\mathcal{Q}\right)\neq \{0\}$. So $q_*$ is nonzero and surjective.

Consider the closed complex analytic subset of the universal moduli of Higgs bundles over ${\mathcal T}_{g,n}$ given by the
kernel of the map $(F,\, \psi)\, \longmapsto\, (\text{trace}(\psi),\, \text{trace}(\psi^2))$
to the universal moduli of forms of degree $1$ and $2$. The ${\mathcal Y''}$ defined as in the statement of the
theorem is the intersection of this closed subset with leaf of the isomonodromic deformation.
Hence ${\mathcal Y''}$ is a closed complex analytic subset of ${\mathcal T}_{g,n}$.
We may assume that in a neighborhood of $t_0\in \mathcal{Y''}$, the non-zero nilpotent Higgs field $\theta$ on $\Epar$ extends to a non-zero nilpotent Higgs field in this neighborhood. 
 Similarly to the proof of Proposition \ref{prop1}, the composition 
$$ T_{t_0}\mathcal{Y}'' \, \hookrightarrow \, \Hun \left(X,\TXD \right)\, \stackrel{q_*}\longrightarrow \,\Hun\left(X,\mathcal{Q}\right)$$
vanishes identically because $\gamma_1\circ \mu_1\,=\,0$. Therefore $\mathcal{Y}'' \neq {\mathcal T}_{g,n}$. 
\end{proof}

\begin{remark}
In the higher rank case, not only does the deformation theory of nilpotent Higgs bundles get much more complicated, 
but the main argument in the proof of Theorem \ref{thm2} breaks down: in arbitrary rank the quotient $\mathcal{Q}$ is not 
necessarily a line bundle and we would need additional information to ensure that $q_*$ is surjective.
\end{remark}

\section*{Acknowledgements}

We thank the referee for comments. The first author is supported by a J. C. Bose Fellowship. The second author 
is supported by ANR-16-CE40-0008.

%%%%%%%%%%%%%%%%%%%%%%%%%%%%%%%%%%%%%%%%%%%%%%%%%%%%%%%%%%%%%%%%%%%

\end{document}